\documentclass[11pt, a4paper]{article}
\usepackage{}

\usepackage{graphicx}
\usepackage{ae}
\usepackage{amsmath}
\usepackage{amssymb}
\usepackage{latexsym}
\usepackage{url}
\usepackage{epsfig}
\usepackage{lineno}
\usepackage{cite}
\usepackage{mathrsfs}
\usepackage{amsfonts}
\usepackage{amsthm}
\usepackage{indentfirst}

\input amssym.def
\input amssym.tex

\def\De{\Delta}

\def\l{\lambda}

\def\ve{\varepsilon}

\def\qed{\hfill$\Box$\vspace{12pt}}

\usepackage{color}

\long\def\delete#1{}

\newcommand{\bmat}[1]{\begin{bmatrix}#1\end{bmatrix}}
\newcommand{\pmat}[1]{\begin{pmatrix}#1\end{pmatrix}}

\newcommand{\be}{\begin{equation}}
\newcommand{\ee}{\end{equation}}
\newcommand{\ben}{\begin{equation*}}
\newcommand{\een}{\end{equation*}}
\newcommand{\bea}{\begin{eqnarray}}
\newcommand{\eea}{\end{eqnarray}}
\newcommand{\bean}{\begin{eqnarray*}}
\newcommand{\eean}{\end{eqnarray*}}

\newtheorem{thm}{Theorem}[section]
\newtheorem{cor}[thm]{Corollary}

\newtheorem{exam}[thm]{Example}

\newtheorem{defn}[thm]{Definition}

\newtheorem{rem}{Remark}
\numberwithin{equation}{section}

\title{Spectra of subdivision-vertex and subdivision-edge neighbourhood coronae}

\author{Xiaogang Liu\\
{\small Department of Mathematics and Statistics}\\[-0.8ex]
{\small The University of Melbourne}\\[-0.8ex]
{\small Parkville, VIC 3010, Australia}\\
\emph{{\small \tt xiaogliu.yzhang@gmail.com}}
\and\; Pengli Lu
\\
{\small School of Computer and Communication}\\[-0.8ex]
{\small Lanzhou University of Technology}\\[-0.8ex]
{\small Lanzhou, 730050, Gansu, P.R. China}\\
\emph{{\small \tt lupengli88@163.com}} }

\date{}

\begin{document}

\openup 0.5\jot
\maketitle

\begin{abstract}
Let $G=(V(G),E(G))$ be a graph with vertex set $V(G)$ and edge set $E(G)$. The subdivision graph $\mathcal{S}(G)$ of a graph $G$ is the graph obtained by inserting a new vertex into every edge of $G$. Let $G_1$ and $G_2$ be two vertex disjoint graphs. The \emph{subdivision-vertex neighbourhood corona} of $G_1$ and $G_2$, denoted by $G_1 \boxdot G_2$, is the graph obtained from $\mathcal{S}(G_1)$ and $|V(G_1)|$ copies of $G_2$, all vertex disjoint, and joining the neighbours of the $i$th vertex of $V(G_1)$ to every vertex in the $i$th copy of $G_2$. The \emph{subdivision-edge neighbourhood corona} of $G_1$ and $G_2$, denoted by $G_1  \boxminus   G_2$, is the graph obtained from $\mathcal{S}(G_1)$ and $|I(G_1)|$ copies of $G_2$, all vertex disjoint, and joining the neighbours of the $i$th vertex of $I(G_1)$ to every vertex in the $i$th copy of $G_2$, where $I(G_1)$ is the set of inserted vertices of $\mathcal{S}(G_1)$. In this paper we determine the adjacency spectra, the Laplacian spectra and the signless Laplacian spectra of $G_1\boxdot G_2$ (respectively, $G_1\boxminus G_2$) in terms of the corresponding spectra of $G_1$ and $G_2$. As applications, these results enable us to construct infinitely many pairs of cospectral graphs, and using the results on the Laplacian spectra of subdivision-vertex neighbourhood coronae, new families of expander graphs are constructed from known ones.

\bigskip


\bigskip

\end{abstract}

\section{Introduction}

All graphs considered in this paper are undirected and simple. Let $G=(V(G),E(G))$ be a
graph with vertex set $V(G)=\{v_1,v_2,\ldots,v_n\}$ and edge set
$E(G)$. The \emph{adjacency matrix} of $G$, denoted by $A(G)$, is the $n \times n$ matrix whose $(i,j)$-entry is $1$ if $v_i$ and $v_j$ are adjacent in $G$ and $0$ otherwise. Denote by $d_i=d_G(v_i)$ the degree of $v_i$ in $G$, and define $D(G)$ to be the diagonal matrix with diagonal entries $d_1,d_2,\ldots,d_n$. The \emph{Laplacian matrix} of $G$ and the \emph{signless Laplacian matrix} of $G$ are defined as $L(G)=D(G)-A(G)$ and $Q(G)=D(G)+A(G)$, respectively.  Given an $n \times n$ matrix $M$, denote by
$$
\phi(M;x)=\det(xI_n-M),
$$
or simply $\phi(M)$, the characteristic polynomial of $M$, where $I_n$ is the identity matrix of size $n$. In particular, for a graph $G$, we call $\phi(A(G))$  (respectively, $\phi(L(G))$, $\phi(Q(G))$) the \emph{adjacency} (respectively, \emph{Laplacian}, \emph{signless Laplacian}) \emph{characteristic polynomial} of $G$, and its roots the \emph{adjacency} (respectively, \emph{Laplacian}, \emph{signless Laplacian}) \emph{eigenvalues} of $G$. Denote the eigenvalues of $A(G), L(G)$ and $Q(G)$, respectively, by $\lambda_1(G)\geq\lambda_2(G)\geq\cdots\geq\lambda_n(G)$, $0=\mu_1(G)\leq\mu_2(G)\leq\cdots\leq\mu_n(G)$, $\nu_1(G)\leq\nu_2(G)\leq\cdots\leq\nu_n(G)$.  The collection of eigenvalues of $A(G)$ together with their multiplicities are called the \emph{$A$-spectrum} of $G$. Two graphs are said to be \emph{$A$-cospectral} if they have the same $A$-spectrum. Similar terminology will be used for $L(G)$ and $Q(G)$. So we can speak of \emph{$L$-spectrum}, \emph{$Q$-spectrum},  \emph{$L$-cospectral graphs} and \emph{$Q$-cospectral graphs}. It is well known that graph spectra store a lot of structural information about a graph; see \cite{kn:Cvetkovic95,kn:Cvetkovic10,kn:Brouwer12} and the references therein.

Until now, many graph operations such as the disjoint union, the Cartesian product, the Kronecker product, the corona, the edge corona and the neighborhood corona have been introduced, and their spectra are computed in \cite{kn:Brouwer12,kn:Cui12,kn:Cvetkovic95,kn:Cvetkovic10,kn:Gopalapillai11,kn:Hou10,kn:McLeman11,kn:Wang12}, respectively. It is well known \cite{kn:Cvetkovic10} that the \emph{subdivision graph} $\mathcal{S}(G)$ of a graph $G$ is the graph obtained by inserting a new vertex into every edge of $G$. We denote the set of such new vertices by $I(G)$.  In \cite{kn:Indulal12}, two new graph operations based on subdivision graphs: \emph{subdivision-vertex join} and \emph{subdivision-edge join}, are introduced, and their $A$-spectra are investigated respectively. Further works on their $L$-spectra and $Q$-spectra are given in \cite{kn:Liu12}. In \cite{kn:Lu12},  the spectra of the so-called \emph{subdivision-vertex corona} and \emph{subdivision-edge corona} are computed, respectively.  Motivated by these works, we define two new graph operations based on subdivision graphs as follows.

\begin{defn}
\label{SVNCdf1}
{\em The \emph{subdivision-vertex neighbourhood corona} of $G_1$ and $G_2$, denoted by $G_1 \boxdot G_2$, is the graph obtained from $\mathcal{S}(G_1)$ and $|V(G_1)|$ copies of $G_2$, all vertex-disjoint, and joining the neighbours of the $i$th vertex of $V(G_1)$ to every vertex in the $i$th copy of $G_2$.}
\end{defn}

\begin{defn}
\label{SENCdf2}
{\em The \emph{subdivision-edge neighbourhood corona} of $G_1$ and $G_2$, denoted by $G_1  \boxminus   G_2$, is the graph obtained from $\mathcal{S}(G_1)$ and $|I(G_1)|$ copies of $G_2$, all vertex-disjoint, and joining the neighbours of the $i$th vertex of $I(G_1)$ to every vertex in the $i$th copy of $G_2$.}
\end{defn}

Note that if $G_1$ and $G_2$ are two graphs on disjoint sets of $n_1$ and $n_2$ vertices, $m_1$ and $m_2$ edges, respectively, then $G_1\boxdot G_2$ has $n_1+m_1+n_1n_2$ vertices, $2m_1+n_1m_2+2m_1n_2$ edges, and $G_1\boxminus G_2$ has $n_1+m_1+m_1n_2$ vertices, $2m_1+m_1m_2+2m_1n_2$ edges.

\begin{exam}
{\em Let $P_n$ denote a path of order $n$. Figure \ref{fff2} depicts the subdivision-vertex neighbourhood corona $P_4\boxdot P_2$ and subdivision-edge neighbourhood corona $P_4\boxminus P_2$, respectively.}
\end{exam}
\begin{figure}[here]
\centering
\vspace{-1.5cm}
\includegraphics*[height=7.8cm]{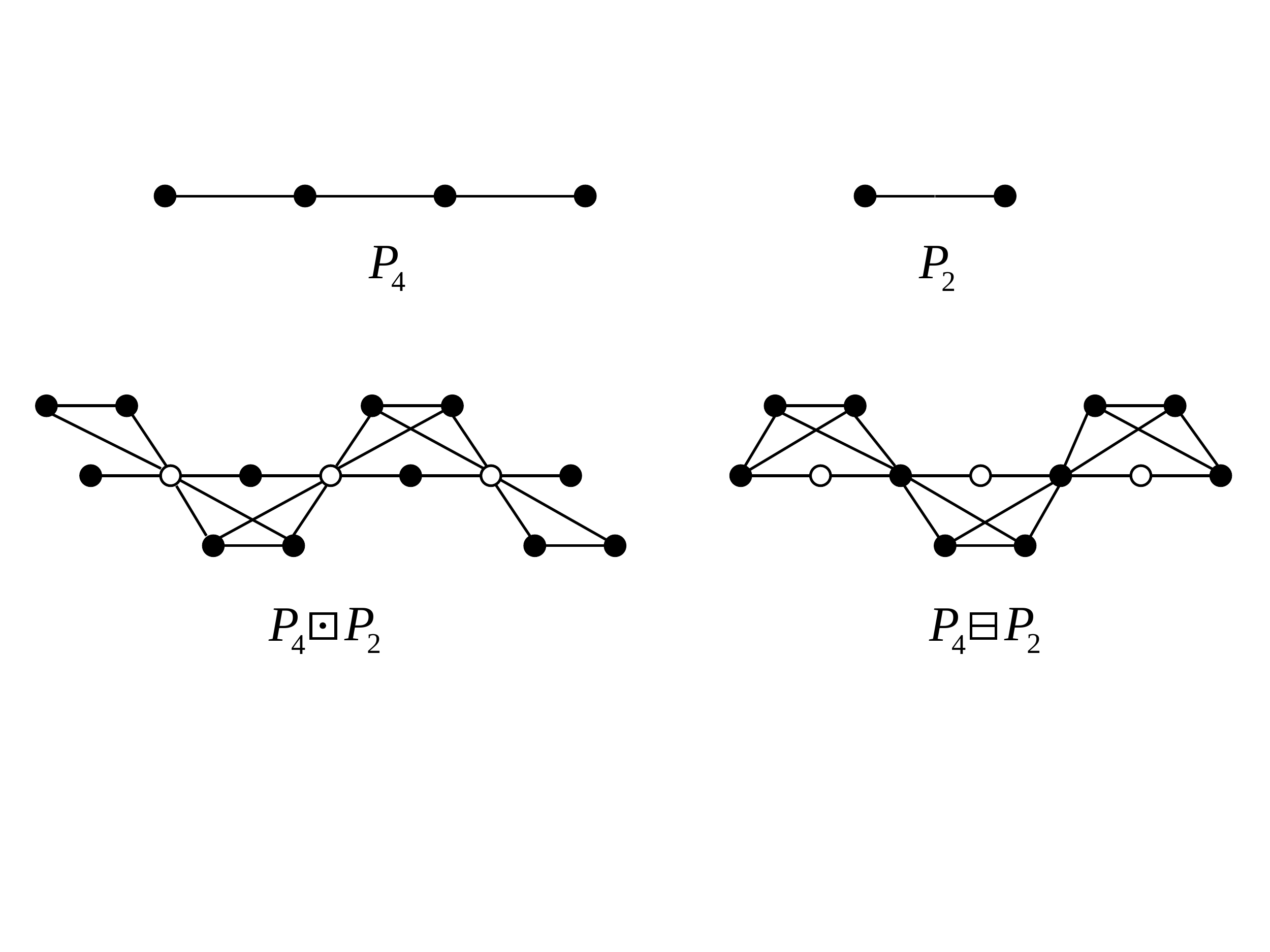}
\vspace{-2.6cm}
\caption{\small An example of subdivision-vertex and subdivision-edge neighbourhood coronae.}
\label{fff2}
\end{figure}
In this paper,  we will determine the $A$-spectra, the $L$-spectra and the $Q$-spectra of $G_1\boxdot G_2$ (respectively, $G_1\boxminus  G_2$) for a regular graph $G_1$ and an arbitrary graph $G_2$ in terms of that of $G_1$ and $G_2$ (see Theorems \ref{SVNCATh1}, \ref{SVNCLthm1}, \ref{SVNCQthm1}, \ref{SENCATh1}, \ref{SENCLTh1} and \ref{SENCQthm1}). As we will see in Corollaries \ref{SVNCAcosp}, \ref{SVNCLcosp}, \ref{SVNCQcosp}, \ref{SENCAcosp}, \ref{SENCLcosp} and \ref{SENCQcosp}, our results on the spectra of $G_1\boxdot  G_2$ and $G_1\boxminus  G_2$ enable us to construct infinitely many pairs of cospectral graphs. Moreover, as stated in Corollary \ref{SVNCExp}, by using the results on the $L$-spectra of subdivision-vertex neighbourhood coronae, we can construct new families of expander graphs from known ones.

\section{Spectra of subdivision-vertex neighbourhood coronae}\label{SVNC:spec}

In this section, we determine the spectra of subdivision-vertex neighbourhood coronae with the help of the \emph{coronal} of a matrix. The \emph{$M$-coronal} $\Gamma_M(x)$ of an $n\times n$ square matrix $M$ is defined \cite{kn:McLeman11,kn:Cui12} to be the sum of the entries of the matrix $(x I_n-M)^{-1}$, that is,
$$\Gamma_M(x)=\textbf{1}_n^T(x I_n-M)^{-1}\textbf{1}_n,$$
where $\textbf{1}_n$ denotes the column vector of size $n$ with all the entries equal one.

It is known \cite[Proposition 2]{kn:Cui12} that, if $M$ is an $n \times n$ matrix with each row sum equal to a constant $t$, then
\be
\label{eq:GammaT}
\Gamma_{M}(x) = \frac{n}{x-t}.
\ee
In particular, since for any graph $G$ with $n$ vertices, each row sum of $L(G)$ is equal to $0$, we have
\be
\label{eq:GammaTL}
\Gamma_{L(G)}(x) = \frac{n}{x}.
\ee

Let $M_1$, $M_2$, $M_3$ and $M_4$ be respectively $p\times p$, $p\times q$, $q\times p$ and $q\times q$ matrices with $M_1$ and $M_4$ invertible. It is well known that
\begin{eqnarray}
  \det\pmat{M_1 & M_2\\
            M_3 & M_4} &=& \det(M_4)\cdot\det\left(M_1-M_2M_4^{-1}M_3\right), \label{schur1}   \\
   &=& \det(M_1)\cdot\det\left(M_4-M_3M_1^{-1}M_2\right) \label{schur2},
\end{eqnarray}
where $M_1-M_2M_4^{-1}M_3$ and $M_4-M_3M_1^{-1}M_2$ are called the \emph{Schur complements} \cite{kn:Schur} of $M_4$ and $M_1$, respectively.

The \emph{Kronecker product} $A\otimes B$ of two matrices $A=(a_{ij})_{m \times n}$ and $B=(b_{ij})_{p \times q}$ is the $mp \times nq$ matrix obtained from $A$ by replacing each element $a_{ij}$ by $a_{ij}B$. This is an associative operation with the property that $(A\otimes B)^T=A^T\otimes B^T$ and $(A\otimes B)(C\otimes D)=AC\otimes BD$ whenever the products $AC$ and $BD$ exist. The latter implies $(A\otimes B)^{-1}=A^{-1}\otimes B^{-1}$ for nonsingular matrices $A$ and $B$. Moreover, if $A$ and $B$ are $n \times n$ and $p \times p$ matrices, then $\det(A\otimes B)=(\det A)^p (\det B)^n$. The reader is referred to \cite{kn:Kronecker} for other properties of the Kronecker product not mentioned here.

Let $G_1$ and $G_2$ be two graphs on disjoint sets of $n_1$ and $n_2$ vertices, $m_1$ and $m_2$ edges, respectively. We first label the vertices of $G_1\boxdot G_2$ as follows. Let $V(G_1)=\left\{v_1,v_2,\ldots,v_{n_1}\right\}$, $I(G_1)=\left\{e_1,e_2,\ldots,e_{m_1}\right\}$ and $V(G_2)=\left\{u_1,u_2,\ldots,u_{n_2}\right\}$. For $i = 1, 2, \ldots, n_1$, let $V^i(G_2)=\{u_1^i,u_2^i,\ldots,u_{n_2}^i\}$ denote the vertex set of the $i$th copy of $G_2$. Then
\be\label{SVNCpart}
V(G_1)\cup I(G_1) \cup \left[V^1(G_2)\cup V^2(G_2)\cup\cdots\cup V^{n_1}(G_2)\right]
\ee
is a partition of $V(G_1\boxdot G_2)$.
Clearly, the degrees of the vertices of $G_1\boxdot  G_2$ are:
\bean
d_{G_1\boxdot  G_2}(v_i) &=& d_{G_1}(v_i),\;\, i=1,2,\ldots,n_1,  \\
d_{G_1\boxdot  G_2}(e_i) &=& 2+2n_2,\;\, i=1,2,\ldots,m_1, \\
d_{G_1\boxdot  G_2}(u_j^i) &=&d_{G_2}(u_j)+d_{G_1}(v_i),\;\, i=1,2,\ldots,n_1,\, j=1,2,\ldots,n_2.
\eean

\subsection{$A$-spectra of subdivision-vertex neighbourhood coronae}
\label{sec:SVNCA}

First, we compute the $A$-spectra of  subdivision-vertex neighbourhood coronae. Before proceeding, we need to mention two basic definitions. The \emph{vertex-edge incidence matrix} $R(G)$ \cite{kn:Godsil01} of a graph $G$ is the $(0,1)$-matrix with rows and columns indexed by the vertices and edges of $G$, respectively, such that the $ve$-entry of $R(G)$ is equal to $1$ if and only if the vertex $v$ is in the edge $e$. The \emph{line graph} \cite{kn:Godsil01} of a graph $G$ is the graph $\mathcal {L}(G)$ with the edges of $G$ as its vertices, and where two edges of $G$ are adjacent in $\mathcal {L}(G)$ if and only if they are incident in $G$.  It is well known \cite{kn:Cvetkovic10} that
\begin{equation}\label{Inc-Line}
R(G)^TR(G)=A(\mathcal {L}(G))+2I_{m},
\end{equation}
where $m$ is the number of edges of $G$.

\begin{thm}\label{SVNCATh1}
Let $G_1$ be an $r_1$-regular graph with $n_1$ vertices and $m_1$ edges, and $G_2$ an arbitrary graph on $n_2$ vertices. Then
\begin{eqnarray*}
\phi\left(A(G_1\boxdot G_2);x\right)=x^{m_1-n_1}\cdot\big(\phi\left(A(G_2);x\right)\big)^{n_1}\cdot\prod_{i=1}^{n_1}\left(x^2-\left(1+x\cdot \Gamma_{A(G_2)}(x)\right) (\lambda_i(G_1)+r_1)\right). \end{eqnarray*}
\end{thm}

\begin{proof}
Let $R$ be the vertex-edge incidence matrix of $G_1$. Then, with respect to the partition (\ref{SVNCpart}), the adjacency matrix of $G_1\boxdot G_2$ can be written as
\[A(G_1\boxdot G_2)=\bmat{
                          0_{n_1\times n_1} & R & 0_{n_1\times n_1}\otimes\textbf{1}_{n_2}^T \\[0.2cm]
                          R^T & 0_{m_1\times m_1} & R^T\otimes\textbf{1}_{n_2}^T\\[0.2cm]
                          \left(0_{n_1\times n_1}\otimes\textbf{1}_{n_2}^T\right)^T & \left(R^T\otimes\textbf{1}_{n_2}^T\right)^T &I_{n_1}\otimes A(G_2)
                       },\]
where  $0_{s\times t}$ denotes the $s\times t$ matrix with all entries equal to zero.
Thus the adjacency characteristic polynomial of $G_1\boxdot G_2$ is given by
\begin{eqnarray*}
\phi\left(A(G_1\boxdot G_2)\right)
&=& \det\bmat{
                          xI_{n_1} & -R &  0_{n_1\times n_1}\otimes\textbf{1}_{n_2}^T \\[0.2cm]
                          -R^T & xI_{m_1} & -R^T\otimes\textbf{1}_{n_2}^T\\[0.2cm]
                            0_{n_1\times n_1}\otimes\textbf{1}_{n_2}  & -R\otimes\textbf{1}_{n_2}  &I_{n_1}\otimes (xI_{n_2}-A(G_2))
                       }\\ [0.2cm]
&=& \det(I_{n_1}\otimes(xI_{n_2}-A(G_2)))\cdot\det(S)\\
&=&\big(\phi\left(A(G_2)\right)\big)^{n_1}\cdot\det(S),
\end{eqnarray*}
where
\begin{eqnarray*}
S&=&\pmat{
         xI_{n_1}  & -R \\[0.2cm]
         -R^T & xI_{m_1}-\Gamma_{A(G_2)}(x)R^TR
       }
\end{eqnarray*}
is the Schur complement of $I_{n_1}\otimes(xI_{n_2}-A(G_2))$ obtained by applying (\ref{schur1}). It is known that \cite[Theorem 2.4.1]{kn:Cvetkovic10} the adjacency eigenvalues of $\mathcal {L}(G_1)$ are $\lambda_i(G_1)+r_1-2$, for $i=1,2,\ldots,n_1$, and $-2$ repeated $m_1-n_1$ times.  Thus, by applying (\ref{Inc-Line}) and (\ref{schur2}), the result follows from
\begin{eqnarray*}
\det (S)&=&x^{n_1}\cdot\det\left(xI_{m_1}-\Gamma_{A(G_2)}(x)R^TR-\frac{1}{x}R^TR\right) \\
&=& x^{n_1}\cdot\prod_{i=1}^{m_1}\left(x-\left(\frac{1}{x}+\Gamma_{A(G_2)}(x)\right) \Big(2+\lambda_i(\mathcal {L}(G_1))\Big)\right) \\
&=& x^{m_1}\cdot\prod_{i=1}^{n_1}\left(x-\left(\frac{1}{x}+\Gamma_{A(G_2)}(x)\right) \Big(\lambda_i(G_1)+r_1\Big)\right).
\end{eqnarray*}\qed
\end{proof}

Theorem \ref{SVNCATh1} implies the following result.

\begin{cor}\label{SVNCAreg}
Let $G_1$ be an $r_1$-regular graph on $n_1$ vertices and $m_1$ edges, and $G_2$ an $r_2$-regular graph on $n_2$ vertices. Then the $A$-spectrum of $G_1\boxdot G_2$ consists of:
\begin{itemize}
  \item[\rm (a)] $0$, repeated $m_1-n_1$ times;
    \item[\rm (b)] $\lambda_i(G_2)$, repeated $n_1$ times, for each $i=2,3,\ldots,n_2$;
  \item[\rm (c)] three roots of the equation
\[x^3-r_2x^2-(1+n_2)(\lambda_j(G_1)+r_1)x+r_2(\lambda_j(G_1)+r_1)=0\]
for each $j=1, 2, \ldots, n_1$.
\end{itemize}
\end{cor}

\begin{proof}
By Theorem \ref{SVNCATh1}, we obtain (a) readily. Since $G_2$ is $r_2$-regular with $n_2$ vertices, by (\ref{eq:GammaT}), we have
$$\Gamma_{A(G_2)}(x) = \frac{n_2}{x-r_2}.$$
The pole of $\Gamma_{A(G_2)}(x)$ is $x=r_2=\l_1(G_2)$. Thus, by Theorem \ref{SVNCATh1}, for each $i=2,3,\ldots,n_2$, $\lambda_i(G_2)$ is an eigenvalue of $G_1\boxdot G_2$ repeated $n_1$ times. The remaining $3n_1$ eigenvalues of $G_1\boxdot G_2$ are obtained by solving
\[x^2-\left(1+ x\cdot\frac{n_2}{x-r_2}\right) (\lambda_j(G_1)+r_1)=0\]
for each $j=1,2,\ldots,n_1$, and this yields the eigenvalues in (c).
\qed
\end{proof}

Denote by $K_{p,q}$ the complete bipartite graph with $p, q \ge 1$ vertices in the two parts of its bipartition. It is known \cite{kn:McLeman11} that the $A(K_{p,q})$-coronal of $K_{p,q}$ is given by
$$
\Gamma_{A(K_{p,q})}(x) = \frac{(p+q)x+2pq}{x^2-pq}.
$$
The $A$-spectrum of $K_{p,q}$ \cite{kn:Brouwer12,kn:Cvetkovic10} consists of $\pm\sqrt{pq}$ with multiplicity one, and $0$ with multiplicity $p+q-2$. Since $\pm\sqrt{pq}$ are the poles of $\Gamma_{A(K_{p,q})}(x)$, Theorem \ref{SVNCATh1} implies the following result immediately. Note that the case when $p = q$ is also covered by Corollary \ref{SVNCAreg}.

\begin{cor}\label{SVNCAComp}
Let $G$ be an $r$-regular graph on $n$ vertices and $m$ edges with $m\ge n$, and let $p, q \ge 1$ be integers. Then the $A$-spectrum of $G \boxdot K_{p,q}$ consists of:
\begin{itemize}
  \item[\rm (a)] $0$, repeated $m+(p+q-3)n$ times;
  \item[\rm (b)] four roots of the equation \[x^4-\big[pq+(1+p+q)(\lambda_j(G)+r)\big]x^2-2pq(\lambda_j(G)+r)x+pq(\lambda_j(G)+r)=0\]
  for each $j=1, 2, \ldots, n_1$.
\end{itemize}
\end{cor}

Until now, many infinite families of pairs of $A$-cospectral graphs are generated by using graph operations (for example, \cite{kn:Barik07,kn:Liu12,kn:McLeman11}). Now, as stated in the following corollary of Theorem \ref{SVNCATh1}, the subdivision-vertex neighborhood corona enables us to obtain infinitely many pairs of $A$-cospectral graphs.

\begin{cor}
\label{SVNCAcosp}
\begin{itemize}
  \item[\rm (a)]  If $G_1$ and $G_2$ are $A$-cospectral regular graphs, and $H$ is any graph, then $G_1\boxdot  H$ and $G_2\boxdot  H$ are $A$-cospectral.
  \item[\rm (b)]  If $G$ is a regular graph, and $H_1$ and $H_2$ are $A$-cospectral graphs with $\Gamma_{A(H_1)}(x)=\Gamma_{A(H_2)}(x)$, then $G\boxdot  H_1$ and $G\boxdot  H_2$ are $A$-cospectral.
\end{itemize}
\end{cor}

\subsection{$L$-spectra of subdivision-vertex neighbourhood coronae}
\label{sec:SVNCL}

\begin{thm}\label{SVNCLthm1}
Let $G_1$ be an $r_1$-regular graph on $n_1$ vertices and $m_1$ edges, and $G_2$ an arbitrary graph on $n_2$ vertices. Then
\begin{eqnarray*}
\phi\left(L(G_1\boxdot  G_2);x\right)&=&x(x-2-2n_2-r_1)(x-2-2n_2)^{m_1-n_1}\cdot\prod_{i=1}^{n_2}(x-r_1-\mu_i(G_2))^{n_1}\\
&&\cdot\prod_{i=2}^{n_1}\left(x^2-(2+2n_2+r_1)x+(1+n_2)\mu_i(G_1)\right).
\end{eqnarray*}
\end{thm}
\begin{proof}
Let $R$ be the vertex-edge incidence matrix of $G_1$. Then, with respect to the partition (\ref{SVNCpart}), the Laplacian matrix of $G_1\boxdot G_2$ can be written as
\[L(G_1\boxdot G_2)=\bmat{
                          r_1I_{n_1} & -R & 0_{n_1\times n_1}\otimes\textbf{1}_{n_2}^T \\[0.2cm]
                          -R^T & (2+2n_2)I_{m_1} & -R^T\otimes\textbf{1}_{n_2}^T\\[0.2cm]
                          \left(0_{n_1\times n_1}\otimes\textbf{1}_{n_2}^T\right)^T & -\left(R^T\otimes\textbf{1}_{n_2}^T\right)^T &I_{n_1}\otimes(r_1I_{n_2}+L(G_2))}.\]
Thus the Laplacian characteristic polynomial of $G_1\boxdot G_2$ is given by
\begin{eqnarray*}
\phi\left(L(G_1\boxdot G_2)\right)
&=& \det\bmat{
                          (x-r_1)I_{n_1} &  R & 0_{n_1\times n_1}\otimes\textbf{1}_{n_2}^T \\[0.2cm]
                          R^T & (x-2-2n_2)I_{m_1} &  R^T\otimes\textbf{1}_{n_2}^T\\[0.2cm]
                          0_{n_1\times n_1}\otimes\textbf{1}_{n_2}  &  R\otimes\textbf{1}_{n_2} & I_{n_1}\otimes((x-r_1)I_{n_2}-L(G_2))}\\ [0.2cm]
&=& \det(I_{n_1}\otimes((x-r_1)I_{n_2}-L(G_2)))\cdot\det(S)\\
&=&\det(S)\cdot\prod_{i=1}^{n_2}(x-r_1-\mu_i(G_2))^{n_1},
\end{eqnarray*}
where
\begin{eqnarray*}
S&=&\pmat{
         (x-r_1)I_{n_1}  &  R \\[0.2cm]
          R^T & (x-2-2n_2)I_{m_1}-\Gamma_{L(G_2)}(x-r_1)R^TR
       }
\end{eqnarray*}
is the Schur complement of $I_{n_1}\otimes((x-r_1)I_{n_2}-L(G_2))$ obtained by applying (\ref{schur1}). Note that $\mu_i(G_1)=r_1-\lambda_i(G_1)$, $i=1,2,\ldots,n_1$. Thus, by (\ref{eq:GammaTL}) and (\ref{schur2}), the result follows from
\begin{eqnarray*}
\det (S)&=&(x-r_1)^{n_1}\cdot\det\left((x-2-2n_2)I_{m_1}-\Gamma_{L(G_2)}(x-r_1)R^TR-\frac{1}{x-r_1}R^TR\right) \\
&=& (x-r_1)^{n_1}\cdot\prod_{i=1}^{m_1}\left(x-2-2n_2-\left(\frac{1}{x-r_1}+\Gamma_{L(G_2)}(x-r_1)\right) \Big(2+\lambda_i(\mathcal {L}(G_1))\Big)\right) \\
&=& (x-2-2n_2)^{m_1-n_1}\prod_{i=1}^{n_1}\big[x^2-(2+2n_2+r_1)x+(1+n_2)\mu_i(G_1)\big].
\end{eqnarray*}
\qed\end{proof}

Theorem \ref{SVNCLthm1} implies the following result.

\begin{cor}
\label{SVNCLcosp}
\begin{itemize}
  \item[\rm (a)]  If $G_1$ and $G_2$ are $L$-cospectral regular graphs, and $H$ is an arbitrary graph, then $G_1\boxdot  H$ and $G_2\boxdot  H$ are $L$-cospectral.
  \item[\rm (b)]  If $G$ is a regular graph, and $H_1$ and $H_2$ are $L$-cospectral graphs, then $G\boxdot  H_1$ and $G\boxdot  H_2$ are $L$-cospectral.
  \item[\rm (c)]  If $G_1$ and $G_2$ are $L$-cospectral regular graphs, and $H_1$ and $H_2$ are $L$-cospectral graphs, then $G_1\boxdot  H_1$ and $G_2\boxdot  H_2$ are $L$-cospectral.
\end{itemize}
\end{cor}

Let $t(G)$ denote the number of spanning trees of $G$. It is well known \cite{kn:Cvetkovic95} that if $G$ is a connected graph on $n$ vertices with Laplacian spectrum $0=\mu_1(G)<\mu_{2}(G)\le\cdots\le\mu_n(G)$, then $$t(G)=\frac{\mu_{2}(G)\cdots\mu_n(G)}{n}.$$ By Theorem \ref{SVNCLthm1}, we can readily obtain the following result.
\begin{cor}
\label{SVNCSptree}
Let $G_1$ be an $r_1$-regular graph on $n_1$ vertices and $m_1$ edges, and $G_2$ an arbitrary graph on $n_2$ vertices. Then
\begin{eqnarray*}
t(G_1\boxdot  G_2)=\frac{(2+2n_2+r_1)\cdot(2+2n_2)^{m_1-n_1}\cdot\prod_{i=1}^{n_2}(r_1+\mu_i(G_2))^{n_1}\cdot\prod_{i=2}^{n_1}(1+n_2)\mu_i(G_1)}{n_1+m_1+n_1n_2}.
\end{eqnarray*}
\end{cor}

It is well known that $a(G) = \mu_2(G)$ is called the \emph{algebraic connectivity} \cite{kn:Fiedler73} of $G$, and $a(G)$ is greater than $0$ if and only if $G$ is a connected graph. Furthermore, $a(G)$ is bounded above by the vertex connectivity of $G$. For more properties on the algebraic connectivity, please refer to \cite{kn:Fiedler73,kn:Cvetkovic10}. An infinite family of graphs, $G_{1}, G_2, G_3, \ldots$, is called a family of \emph{$\ve$-expander graphs} \cite{HLW}, where $\ve > 0$ is a fixed constant, if (i) all these graphs are $k$-regular for a fixed integer $k \ge 3$; (ii) $a(G_i) \ge \ve$ for $i = 1, 2, 3, \ldots$; and (iii) $n_i = |V(G_i)| \rightarrow \infty$ as $i \rightarrow \infty$. (Here we use the algebraic connectivity to define expander families. This is equivalent to the usual definition by using the isoperimetric number $i(G)$, because $a(G)/2 \le i(G) \le \sqrt{a(G)(2\De - a(G))}$ \cite{Moh} (see also \cite[Theorem 7.5.15]{kn:Cvetkovic10}) for any graph $G \ne K_1, K_2, K_3$ with maximum degree $\De$.)

In the following, we will use the subdivision-vertex neighbourhood corona to construct new families of expander graphs from known ones.

Suppose that $G$ is an $r$-regular graph with $r$ even, and $H$ is an edgeless graph with $\frac{r}{2}-1$ vertices. Then $G\boxdot H$ is still an $r$-regular graph. Moreover, Theorem \ref{SVNCLthm1} implies that
$$a(G\boxdot H)=r-\sqrt{r^2-\frac{r}{2}a(G)},$$
since the function $$f(x):=r-\sqrt{r^2-rx/2}$$ strictly increases with $x$ in $[0,2r]$.

Denote the function iteration of $f(x)$ by $f^1(x)=f(x)$ and $f^{j+1}(x)=f(f^j(x))$ for $j\ge1$. Define $G^1_i(H)=G_i\boxdot H$ and $G^{j+1}_i(H)=G_i^j(H)\boxdot H$ for $j\ge1$. Then we have the following result.

\begin{cor}
\label{SVNCExp}
Suppose $G_1, G_2, G_3, \ldots$ is a family of (non-complete) $r$-regular $\ve$-expander graphs with $r$ even. Let $H$ be an edgeless graph with $r/2-1$ vertices. Then $G_1^j(H), G_2^j(H), G_3^j(H), \ldots$ is a family of $r$-regular $f^j(\ve)$-expander graphs.
\end{cor}

\begin{rem}
\emph{It is known \cite[Theorem 7.4.4]{kn:Cvetkovic10} that, if $v_i$ and $v_j$ are two non-adjacent vertices of a graph $G$, then
$$
a(G)\le\frac{1}{2}\left(d_G(v_i)+d_G(v_j)\right).
$$
Thus, we have $0<\varepsilon\le a(G_i)\le r$, which implies that $f(\ve)<\ve$. Therefore, $f^j(\ve)$ decreases as $j$ increases, since $f(x)$ strictly increases with $x$ in $[0,2r]$.}
\end{rem}
\subsection{$Q$-spectra of subdivision-vertex neighbourhood coronae}
\label{sec:SVNCQ}

\begin{thm}\label{SVNCQthm1}
Let $G_1$ be an $r_1$-regular graph on $n_1$ vertices and $m_1$ edges, and $G_2$ an arbitrary graph on $n_2$ vertices. Then
\begin{eqnarray*}
\phi\left(Q(G_1\boxdot  G_2);x\right)&=&(x-2-2n_2)^{m_1-n_1}\cdot\prod_{i=1}^{n_2}(x-r_1-\nu_i(G_2))^{n_1}\\
&& \cdot\prod_{i=1}^{n_1}\Big((x-2-2n_2)(x-r_1)-\big(1+(x-r_1)\Gamma_{Q(G_2)}(x-r_1)\big)\nu_i(G_1)\Big).
\end{eqnarray*}
\end{thm}
\begin{proof}
Let $R$ be the vertex-edge incidence matrix of $G_1$. Then the signless Laplacian matrix of $G_1\boxdot  G_2$ can be written as
\[Q(G_1\boxdot  G_2)=\bmat{
                          r_1I_{n_1} &  R & 0_{n_1\times n_1}\otimes\textbf{1}_{n_2}^T \\[0.2cm]
                           R^T & (2+2n_2)I_{m_1} &  R^T\otimes\textbf{1}_{n_2}^T\\[0.2cm]
                          \left(0_{n_1\times n_1}\otimes\textbf{1}_{n_2}^T\right)^T &  \left(R^T\otimes\textbf{1}_{n_2}^T\right)^T &I_{n_1}\otimes(r_1I_{n_2}+Q(G_2))}.\]
The rest of the proof is similar to that of Theorem \ref{SVNCLthm1} and hence we omit details. \qed\end{proof}

Again, by applying (\ref{eq:GammaT}), Theorem \ref{SVNCQthm1} implies the following result.
\begin{cor}\label{SVNCQcor1}
Let $G_1$ be an $r_1$-regular graph on $n_1$ vertices and $m_1$ edges, and $G_2$ an $r_2$-regular graph on $n_2$ vertices. Then
\begin{eqnarray*}
\phi\left(Q(G_1\boxdot G_2);x\right)&=&(x-2-2n_2)^{m_1-n_1}\cdot\prod_{i=1}^{n_2-1}(x-r_1-\nu_i(G_2))^{n_1}\cdot\prod_{i=1}^{n_1}(a-b\nu_i(G_1)),
\end{eqnarray*}
where $a=(x-2-2n_2)(x-r_1)(x-r_1-2r_2)$ and $b=(1+n_2)x-r_1-2r_2-r_1n_2$.
\end{cor}

Theorem \ref{SVNCQthm1} can enable us to construct infinitely many pairs of $Q$-cospectral graphs.

\begin{cor}
\label{SVNCQcosp}
\begin{itemize}
  \item[\rm (a)]  If $G_1$ and $G_2$ are $Q$-cospectral regular graphs, and $H$ is any graph, then $G_1\boxdot H$ and $G_2\boxdot  H$ are $Q$-cospectral.
  \item[\rm (b)]  If $G$ is a regular graph, and $H_1$ and $H_2$ are $Q$-cospectral graphs with $\Gamma_{Q(H_1)}(x)=\Gamma_{Q(H_2)}(x)$, then $G\boxdot  H_1$ and $G\boxdot  H_2$ are $Q$-cospectral.
\end{itemize}
\end{cor}

\section{Spectra of subdivision-edge neighbourhood coronae}\label{SENC:spec}

In this section, we determine the spectra of subdivision-edge neighbourhood coronae. Let $G_1$ and $G_2$ be two graphs on disjoint sets of $n_1$ and $n_2$ vertices, $m_1$ and $m_2$ edges, respectively. First, we label the vertices of $G_1\boxminus  G_2$ as follows. Let $V(G_1)=\left\{v_1,v_2,\ldots,v_{n_1}\right\}$, $I(G_1)=\left\{e_1,e_2,\ldots,e_{m_1}\right\}$ and $V(G_2)=\left\{u_1,u_2,\ldots,u_{n_2}\right\}$. For $i = 1, 2, \ldots, m_1$, let $V^i(G_2)=\{u_1^i,u_2^i,\ldots,u_{n_2}^i\}$ denote the vertex set of the $i$th copy of $G_2$. Then
\be\label{SENCpart}
V(G_1)\cup I(G_1) \cup \left[V^1(G_2)\cup V^2(G_2)\cup\cdots\cup V^{m_1}(G_2)\right]
\ee
is a partition of $V(G_1\boxminus  G_2)$.
Clearly, the degrees of the vertices of $G_1\boxminus   G_2$ are:
\bean
d_{G_1\boxminus   G_2}(v_i) &=& d_{G_1}(v_i)(1+n_2),\;\, i=1,2,\ldots,n_1, \\
d_{G_1\boxminus   G_2}(e_i) &=& 2,\;\, i=1,2,\ldots,m_1, \\
d_{G_1\boxminus   G_2}(u_j^i) &=&d_{G_2}(u_j)+2,\;\, i=1,2,\ldots,n_1,\, j=1,2,\ldots,n_2.
\eean

\subsection{$A$-spectra of subdivision-edge neighbourhood coronae}
\label{sec:SENCA}

\begin{thm}\label{SENCATh1}
Let $G_1$ be an $r_1$-regular graph with $n_1$ vertices and $m_1$ edges, and $G_2$ an arbitrary graph on $n_2$ vertices. Then
\begin{eqnarray*}
\phi\left(A(G_1\boxminus  G_2);x\right)=x^{m_1-n_1}\cdot\big(\phi\left(A(G_2);x\right)\big)^{m_1}\cdot\prod_{i=1}^{n_1}\left(x^2-\left(1+x\cdot \Gamma_{A(G_2)}(x)\right) (\lambda_i(G_1)+r_1)\right). \end{eqnarray*}
\end{thm}

\begin{proof}
Let $R$ be the vertex-edge incidence matrix of $G_1$. Then, with respect to the partition (\ref{SENCpart}), the adjacency matrix of $G_1\boxminus  G_2$ can be written as
\[A(G_1\boxminus  G_2)=\bmat{
                          0_{n_1\times n_1} & R & R\otimes\textbf{1}_{n_2}^T \\[0.2cm]
                          R^T & 0_{m_1\times m_1} & 0_{m_1\times m_1}\otimes\textbf{1}_{n_2}^T\\[0.2cm]
                          \left(R\otimes\textbf{1}_{n_2}^T\right)^T & \left(0_{m_1\times m_1}\otimes\textbf{1}_{n_2}^T\right)^T &I_{m_1}\otimes A(G_2)
                       }.\]
Thus the adjacency characteristic polynomial of $G_1\boxminus  G_2$ is given by
\begin{eqnarray*}
\phi\left(A(G_1\boxminus  G_2)\right)
&=& \det\bmat{
                          xI_{n_1} & -R &  -R\otimes\textbf{1}_{n_2}^T \\[0.2cm]
                          -R^T & xI_{m_1} & 0_{m_1\times m_1}\otimes\textbf{1}_{n_2}^T\\[0.2cm]
                            -R^T\otimes\textbf{1}_{n_2}  &  0_{m_1\times m_1}\otimes\textbf{1}_{n_2} &I_{m_1}\otimes (xI_{n_2}-A(G_2))
                       }\\ [0.2cm]
&=& \det(I_{m_1}\otimes(xI_{n_2}-A(G_2)))\cdot\det(S)\\
&=&\big(\phi\left(A(G_2)\right)\big)^{m_1}\cdot\det(S),
\end{eqnarray*}
where
\begin{eqnarray*}
S&=&\pmat{
         xI_{n_1}-\Gamma_{A(G_2)}(x)RR^T  & -R \\[0.2cm]
         -R^T & xI_{m_1}
       }
\end{eqnarray*}
is the Schur complement of $I_{m_1}\otimes(xI_{n_2}-A(G_2))$ obtained by applying (\ref{schur1}). It is well known \cite{kn:Cvetkovic10} that $RR^T=A(G_1)+r_1I_{n_1}$. Thus, by (\ref{schur1}) again, the result follows from
\begin{eqnarray*}
\det (S)&=&x^{m_1}\cdot\det\left(xI_{n_1}-\Gamma_{A(G_2)}(x)RR^T-\frac{1}{x}RR^T\right) \\
&=& x^{m_1}\cdot\prod_{i=1}^{n_1}\left(x-\left(\frac{1}{x}+\Gamma_{A(G_2)}(x)\right) \Big(\lambda_i(G_1)+r_1\Big)\right).
\end{eqnarray*}\qed
\end{proof}

Similar to Corollaries \ref{SVNCAreg}, \ref{SVNCAComp} and \ref{SVNCAcosp}, Theorem \ref{SENCATh1} implies the following results.

\begin{cor}\label{SENCAreg}
Let $G_1$ be an $r_1$-regular graph on $n_1$ vertices and $m_1$ edges, and $G_2$ an $r_2$-regular graph on $n_2$ vertices. Then the $A$-spectrum of $G_1\boxminus  G_2$ consists of:
\begin{itemize}
  \item[\rm (a)] $0$, repeated $m_1-n_1$ times;
  \item[\rm (b)] $r_2$, repeated $m_1-n_1$ times;
    \item[\rm (c)] $\lambda_i(G_2)$, repeated $m_1$ times, for each $i=2,3,\ldots,n_2$;
  \item[\rm (d)] three roots of the equation
\[x^3-r_2x^2-(1+n_2)(\lambda_j(G_1)+r_1)x+r_2(\lambda_j(G_1)+r_1)=0\]
for each $j=1, 2, \ldots, n_1$.
\end{itemize}
\end{cor}

\begin{cor}\label{SENCAComp}
Let $G$ be an $r$-regular graph on $n$ vertices and $m$ edges with $m\ge n$, and let $p, q \ge 1$ be integers. Then the $A$-spectrum of $G \boxminus  K_{p,q}$ consists of:
\begin{itemize}
  \item[\rm (a)] $0$, repeated $(p+q-1)m-n$ times;
  \item[\rm (b)] $\pm\sqrt{pq}$, repeated $m-n$ times;
  \item[\rm (c)] four roots of the equation \[x^4-\big[pq+(1+p+q)(\lambda_j(G)+r)\big]x^2-2pq(\lambda_j(G)+r)x+pq(\lambda_j(G)+r)=0\]
  for each $j=1, 2, \ldots, n_1$.
\end{itemize}
\end{cor}

\begin{cor}
\label{SENCAcosp}
\begin{itemize}
  \item[\rm (a)]  If $G_1$ and $G_2$ are $A$-cospectral regular graphs, and $H$ is any graph, then $G_1\boxminus   H$ and $G_2\boxminus   H$ are $A$-cospectral.
  \item[\rm (b)]  If $G$ is a regular graph, and $H_1$ and $H_2$ are $A$-cospectral graphs with $\Gamma_{A(H_1)}(x)=\Gamma_{A(H_2)}(x)$, then $G\boxminus  H_1$ and $G\boxminus   H_2$ are $A$-cospectral.
\end{itemize}
\end{cor}

\subsection{$L$-spectra of subdivision-edge neighbourhood coronae}
\label{sec:SENCL}

\begin{thm}\label{SENCLTh1}
Let $G_1$ be an $r_1$-regular graph with $n_1$ vertices and $m_1$ edges, and $G_2$ an arbitrary graph on $n_2$ vertices. Then
\begin{eqnarray*}
\phi\left(L(G_1\boxminus  G_2);x\right)&=& x(x-2-r_1-r_1n_2)(x-2)^{m_1-n_1}\cdot\prod_{i=1}^{n_2}(x-2-\mu_i(G_2))^{m_1}\\
&&\cdot\prod_{i=2}^{n_1}\left(x^2-(2+r_1+r_1n_2)x+(1+n_2)\mu_i(G_1)\right).
\end{eqnarray*}
\end{thm}

\begin{proof}
Let $R$ be the vertex-edge incidence matrix of $G_1$. Then, with respect to the partition (\ref{SENCpart}), the Laplacian matrix of $G_1\boxminus  G_2$ can be written as
\[L(G_1\boxminus  G_2)=\bmat{
                          r_1(1+n_2)I_{n_1} & -R & -R\otimes\textbf{1}_{n_2}^T \\[0.2cm]
                          -R^T & 2I_{m_1} & 0_{m_1\times m_1}\otimes\textbf{1}_{n_2}^T\\[0.2cm]
                          -\left(R\otimes\textbf{1}_{n_2}^T\right)^T & \left(0_{m_1\times m_1}\otimes\textbf{1}_{n_2}^T\right)^T &I_{m_1}\otimes (2I_{n_2}+L(G_2))
                       }.\]
Thus the Laplacian characteristic polynomial of $G_1\boxminus  G_2$ is given by
\begin{eqnarray*}
\phi\left(L(G_1\boxminus  G_2)\right)
&=& \det\bmat{
                          (x-r_1-r_1n_2)I_{n_1} &  R &   R\otimes\textbf{1}_{n_2}^T \\[0.2cm]
                           R^T & (x-2)I_{m_1} & 0_{m_1\times m_1}\otimes\textbf{1}_{n_2}^T\\[0.2cm]
                            R^T\otimes\textbf{1}_{n_2}  &  0_{m_1\times m_1}\otimes\textbf{1}_{n_2} &I_{m_1}\otimes ((x-2)I_{n_2}-L(G_2))
                       }\\ [0.2cm]
&=& \det(I_{m_1}\otimes ((x-2)I_{n_2}-L(G_2)))\cdot\det(S)\\
&=&\det(S)\cdot\prod_{i=1}^{n_2}(x-2-\mu_i(G_2))^{m_1},
\end{eqnarray*}
where
\begin{eqnarray*}
S&=&\pmat{
         (x-r_1-r_1n_2)I_{n_1}-\Gamma_{L(G_2)}(x-2)RR^T  & R \\[0.2cm]
         R^T & (x-2)I_{m_1}
       }
\end{eqnarray*}
is the Schur complement of $I_{m_1}\otimes ((x-2)I_{n_2}-L(G_2))$ obtained by applying (\ref{schur1}). Thus, by (\ref{schur1}) again, and (\ref{eq:GammaTL}), the result follows from
\begin{eqnarray*}
\det (S)&=&(x-2)^{m_1}\cdot\det\left((x-r_1-r_1n_2)I_{n_1}-\Gamma_{L(G_2)}(x-2)RR^T-\frac{1}{x-2}RR^T\right) \\
&=& (x-2)^{m_1}\cdot\prod_{i=1}^{n_1}\left(x-r_1-r_1n_2-\frac{n_2+1}{x-2}\cdot(2r_1-\mu_i(G_1))\right)\\
&=& (x-2)^{m_1-n_1}\cdot\prod_{i=1}^{n_1}\left(x^2-(2+r_1+r_1n_2)x+(1+n_2)\mu_i(G_1)\right).
\end{eqnarray*}\qed
\end{proof}

Theorem \ref{SENCLTh1} implies the following results.
\begin{cor}
\label{SENCLcosp}
\begin{itemize}
  \item[\rm (a)]  If $G_1$ and $G_2$ are $L$-cospectral regular graphs, and $H$ is an arbitrary graph, then $G_1\boxminus  H$ and $G_2\boxminus  H$ are $L$-cospectral.
  \item[\rm (b)]  If $G$ is a regular graph, and $H_1$ and $H_2$ are $L$-cospectral graphs, then $G\boxminus  H_1$ and $G\boxminus  H_2$ are $L$-cospectral.
  \item[\rm (c)]  If $G_1$ and $G_2$ are $L$-cospectral regular graphs, and $H_1$ and $H_2$ are $L$-cospectral graphs, then $G_1\boxminus  H_1$ and $G_2 \boxminus H_2$ are $L$-cospectral.
\end{itemize}
\end{cor}

\begin{cor}\label{SENCSptree}
Let $G_1$ be an $r_1$-regular graph on $n_1$ vertices and $m_1$ edges, and $G_2$ an arbitrary graph on $n_2$ vertices. Then
\begin{eqnarray*}
t(G_1 \boxminus   G_2)=\frac{2^{m_1-n_1}\cdot(2+r_1+r_1n_2)\cdot\prod_{i=1}^{n_2}(2+\mu_i(G_2))^{m_1}\cdot\prod_{i=2}^{n_1}(1+n_2)\mu_i(G_1)}{n_1+m_1+m_1n_2}.
\end{eqnarray*}
\end{cor}

\subsection{$Q$-spectra of subdivision-edge neighbourhood coronae}
\label{sec:SENCQ}

\begin{thm}\label{SENCQthm1}
Let $G_1$ be an $r_1$-regular graph on $n_1$ vertices and $m_1$ edges, and $G_2$ an arbitrary graph on $n_2$ vertices. Then
\begin{eqnarray*}
\phi\left(Q(G_1 \boxminus   G_2);x\right) &=& (x-2)^{m_1-n_1}\cdot \prod_{i=1}^{n_2}(x-2-\nu_i(G_2))^{m_1}\\
                                          && \cdot\prod_{i=1}^{n_1}\Big((x-r_1-r_1n_2)(x-2)-\big(1+(x-2)\Gamma_{Q(G_2)}(x-2)\big)\nu_i(G_1)\Big).
\end{eqnarray*}
\end{thm}

\begin{proof}
Let $R$ be the vertex-edge incidence matrix of $G_1$. Then the signless Laplacian matrix of $G_1\boxminus G_2$ can be written as
\[Q(G_1\boxminus  G_2)=\bmat{
                          r_1(1+n_2)I_{n_1} &  R &  R\otimes\textbf{1}_{n_2}^T \\[0.2cm]
                           R^T & 2I_{m_1} & 0_{m_1\times m_1}\otimes\textbf{1}_{n_2}^T\\[0.2cm]
                           \left(R\otimes\textbf{1}_{n_2}^T\right)^T & \left(0_{m_1\times m_1}\otimes\textbf{1}_{n_2}^T\right)^T &I_{m_1}\otimes (2I_{n_2}+Q(G_2))
                       }.\]
The result follows by applying $RR^T=Q(G_1)$ and refining the arguments used to prove Theorem \ref{SENCLTh1}.
\qed
\end{proof}

By applying (\ref{eq:GammaT}), Theorem \ref{SENCQthm1} implies the following result.
\begin{cor}\label{SENCQreg}
Let $G_1$ be an $r_1$-regular graph on $n_1$ vertices and $m_1$ edges, and $G_2$ an $r_2$-regular graph on $n_2$ vertices. Then
\begin{eqnarray*}
\phi\left(Q(G_1\boxminus  G_2);x\right)=(x-2)^{m_1-n_1}\cdot(x-2-2r_2)^{m_1-n_1}\cdot \prod_{i=1}^{n_2-1}(x-2-\nu_i(G_2))^{m_1}\cdot\prod_{i=1}^{n_1}(a-b\nu_i(G_1)),
\end{eqnarray*}
where $a=(x-r_1-r_1n_2)(x-2)(x-2-2r_2)$ and $b=(1+n_2)x-2-2r_2-2n_2$.
\end{cor}

Finally, Theorem \ref{SENCQthm1} enables us to construct infinitely many pairs of $Q$-cospectral graphs.

\begin{cor}
\label{SENCQcosp}
\begin{itemize}
  \item[\rm (a)]  If $G_1$ and $G_2$ are $Q$-cospectral regular graphs, and $H$ is any graph, then $G_1\boxminus  H$ and $G_2\boxminus  H$ are $Q$-cospectral.
  \item[\rm (b)]  If $G$ is a regular graph, and $H_1$ and $H_2$ are $Q$-cospectral graphs with $\Gamma_{Q(H_1)}(x)=\Gamma_{Q(H_2)}(x)$, then $G\boxminus  H_1$ and $G\boxminus  H_2$ are $Q$-cospectral.
\end{itemize}
\end{cor}


\begin{thebibliography}{99}
\setlength{\parskip}{0pt} \addtolength{\itemsep}{-4pt}
\footnotesize{

\bibitem{kn:Barik07}
S. Barik, S. Pati, B. K. Sarma, The spectrum of the corona of two graphs, SIAM J. Discrete Math. 24 (2007) 47--56.

\bibitem{kn:Brouwer12}
A. E. Brouwer, W. H. Haemers, Spectra of Graphs, Springer, 2012.

\bibitem{kn:Cui12}
S.-Y. Cui, G.-X. Tian, The spectrum and the signless Laplacian spectrum of coronae, Linear Algebra Appl. 437 (2012) 1692--1703.

\bibitem{kn:Cvetkovic95}
D. M. Cvetkovi\'{c}, M. Doob, H. Sachs, Spectra of Graphs - Theory and Applications, Third edition, Johann Ambrosius Barth. Heidelberg, 1995.

\bibitem{kn:Cvetkovic10}
D. M. Cvetkovi\'{c}, P. Rowlinson, H. Simi\'{c}, An Introduction to the Theory of Graph Spectra, Cambridge University Press, Cambridge, 2010.

\bibitem{kn:Fiedler73} M. Fiedler, Algebraic connectivity of Graphs, Czechoslovak Mathematical Journal 23 (98) 1973.

\bibitem{kn:Godsil01}
C. Godsil, G. Royle, Algebraic Graph Theory, Springer, New York, 2001.

\bibitem{kn:Gopalapillai11}
I. Gopalapillai, The spectrum of neighborhood corona of graphs, Kragujevac Journal of Mathematics 35 (2011) 493--500.


\bibitem{HLW}
S. Hoory, N. Linial,  A. Wigderson, Expander graphs and their applications,  Bull. Amer. Math. Soc. 43(4)  (2006) 439--561.

\bibitem{kn:Hou10}
Y.-P. Hou, W.-C. Shiu, The spectrum of the edge corona of two graphs, Electron. J. Linear Algebra. 20 (2010) 586--594.


\bibitem{kn:Indulal12}
G. Indulal, Spectrum of two new joins of graphs and infinite families of integral graphs, Kragujevac J. Math. 36 (2012) 133--139.


\bibitem{kn:Kronecker}
R. A. Horn, C. R. Johnson, Topics in Matrix Analysis, Cambridge University Press, 1991.


\bibitem{kn:Liu12} X.-G. Liu, Z.-H. Zhang, Spectra of subdivision-vertex and subdivision-edge joins of graphs, submitted.

\bibitem{kn:Lu12} P.-L. Lu, Y.-F. Miao, Spectra of subdivision-vertex and subdivision-edge coronae, manuscript.


\bibitem{Moh}
B. Mohar, Isoperimetric number of graphs, J. Combin. Theory (B) 47 (1989) 274--291.

\bibitem{kn:McLeman11}
C. McLeman, E. McNicholas, Spectra of coronae, Linear Algebra Appl. 435 (2011) 998--1007.

\bibitem{kn:Wang12}
S.-L. Wang, B. Zhou, The signless Laplacian spectra of the corona and edge corona of two graphs, Linear Multilinear Algebra (2012) 1--8, iFirst.

\bibitem{kn:Schur}
F.-Z. Zhang, The Schur Complement and Its Applications, Springer, 2005.
}
\end{thebibliography}
\end{document}